\documentclass{amsart}
\usepackage{graphicx}
\usepackage{amssymb}
\usepackage{amsmath}
\usepackage{amsthm}
\usepackage{subfigure}
\usepackage{enumerate}
\usepackage{bbm}
\usepackage{calligra}
\DeclareMathAlphabet{\mathcalligra}{T1}{calligra}{m}{n}

\newtheorem{thm}{Theorem}[section]
\newtheorem{cor}[thm]{Corollary}

\newtheorem{lem}[thm]{Lemma}
\newtheorem{prop}[thm]{Proposition}
\theoremstyle{definition}
\newtheorem{defn}[thm]{Definition}
\newtheorem{rem}[thm]{Remark}

\newtheorem*{defn*}{Definition}
\newtheorem*{thm*}{Theorem}
\newtheorem*{rems*}{Remarks}
\newtheorem*{rem*}{Remark}

\numberwithin{equation}{section}

\newcommand{\Eq}{{\text{E}}}

\newcommand{\M}{{M}}

\newcommand{\Cwms}{\text{CWMS}}
\newcommand{\Sms}{\text{SMS}}
\newcommand{\lcm}{\text{lcm}}

\begin{document}

\title[Isoperimetric equalities for rosettes] {Isoperimetric equalities for rosettes}
\author{ Micha\l{} Zwierzy\'nski}
\address{Faculty of Mathematics and Information Science\\
Warsaw University of Technology\\
ul. Koszykowa 75, \\
00-662 Warszawa}

\email{zwierzynskim@mini.pw.edu.pl}



\begin{abstract}

In this paper we study the isoperimetric--type equalities for rosettes, i.e. regular closed planar curves with non--vanishing curvature. We find the exact relations between the length and the oriented area of rosettes based on the oriented areas of the Wigner caustic, the Constant Width Measure Set and the Spherical Measure Set.

We also study and find new results about the geometry of affine equidistants of rosettes and of the union of rosettes.

\end{abstract}

\maketitle

%%%%%%%%%%%%%%%%%%%%%%%%
%%%%% Introduction %%%%%%%%%%%%%
%%%%%%%%%%%%%%%%%%%%%%%%
\section{Introduction}
The \textit{classical isoperimetric inequality} for plane curves states the following.
\begin{thm} (Isoperimetric inequality) Let $M$ be a simple curve. Let $L_M$, $A_M$ be the length of $M$ and the area of region enclosed by $M$, respectively. Then
\begin{align*}
L_{M}^2\geqslant 4\pi A_M,
\end{align*}
and the equality holds if and only if $M$ is a circle.
\end{thm}

This result is known already in the Ancient Greece but first proof was given in the nineteenth century by Steiner \cite{S1}. After this paper, there have been many new proofs, applications and generalizations of this theorem, see  \cite{C1, FKN1, G6, G7, GO1, G4, H3,  L1, S1, Zw1, Zw2} and the literature therein. Recently in \cite{Zw1, Zw2} we prove the improved isoperimetric--type inequalities and equalities for planar ovals.

\begin{thm}\label{ThmIsoEqOvals}\cite{Zw1, Zw2} (Isoperimetric--type equalities)
Let $M$ be an oval. Then
\begin{align*}
L_{M}^2=4\pi A_M+8\pi\left|\widetilde{A}_{E_{0.5}(\M)}\right|+\pi\left|\widetilde{A}_{\Cwms(M)}\right|,
\end{align*}
and if $M$ is not a curve of constant width, then
\begin{align*}
L_{M}^2=4\pi A_M+4\pi\left|\widetilde{A}_{\Sms(M)}\right|,
\end{align*}
where $\widetilde{A}_{E_{0.5}(M)}, \widetilde{A}_{\Cwms(M)}, \widetilde{A}_{\Sms(M)}$ are the oriented area of the Wigner caustic of $M$, the oriented area of the Constant Width Measure Set of $M$ the oriented area of the Spherical Measure Set of $M$, respectively.

Furthermore $M$ is a curve of constant width if and only if 
\begin{align*}
L_{M}^2=4\pi A_M+8\pi\left|\widetilde{A}_{E_{0.5}(M)}\right|,
\end{align*}
and $M$ is a curve which has the center of symmetry if and only if 
\begin{align*}
L_{M}^2=4\pi A_M+\pi\left|\widetilde{A}_{\Cwms(M)}\right|.
\end{align*}
\end{thm}

The \textit{Wigner caustic} of an oval on the affine symplectic plane was first introduced by Berry, in his celebrated 1977 paper \cite{B1} on the semiclassical limit of Wigner's phase--space representation of quantum states. The Wigner caustic is an example of an \textit{affine $\lambda$--equidistant} (for $\displaystyle\lambda=\frac{1}{2}$). The affine $\lambda$--equidistant is the set of points divided chords connecting points on $M$ where tangent lines to $M$ are parallel in the ratio $\lambda$. There are many papers considering affine equidistants, see \cite{C2, DMR1, DR1, DRS1, DZ1, G3, GWZ1, JJR1, RZ1, Z1}. In \cite{C2, G3} the Wigner caustic is known as the \textit{area evolute} and in \cite{JJR1} is known as the \textit{symmetry defect}. The well--known Blaschke--S\"uss theorem (\cite{G1, L2}) states that if $M$ is an oval then $M$ has at least three \textit{antipodal pairs}, i.e. two points where the tangent lines at these points are parallel and the curvatures are equal. Singularities of the Wigner caustic occur exactly from an antipodal pair of an oval. The Wigner caustic also leads to one of two constructions of improper affine spheres (\cite{CDR1}). In \cite{DZ1} we study the global geometry of the Wigner caustic and equidistants of plane curves, we find many global results, for instance we presented an algorithm to describe smooth branches of equidistants of parametrized closed planar curves. By this algorithm we can find among other things the exact number of smooth branches, the number of inflexion points, the parity of the number of cusps of each branch. In this paper we use this method in Section \ref{SecWCUnion}.

The paper is organized as follows.

In Section \ref{SectRosettes} we study the geometry of affine $\lambda$--equidistants, the Wigner caustic, the Constant Width Measure Set and the Spherical Measure Set of rosettes, i.e. closed smooth regular curves with non--vanishing curvature. We prove generalizations of Theorem \ref{ThmIsoEqOvals} for rosettes.

Section \ref{SecWCUnion} contains theorems on the geometry of branches of affine $\lambda$--equidistants of the union of rosettes.

%%%%%%%%%%%%%%%%%%%%%%%%%%%%%%%%%%%%%%%%%%%%%%%%%%%
%%%%%%%%%%%%%%%%%%%%%%%%%%%%%%%%%%%%%%%%%%%%%%%%%%%%
%%%%%%%%%%%%%%%%%%%%%%%%%%%%%%%%%%%%%%%%%%%%%%%%%%%%

\section{Rosettes}\label{SectRosettes}

Rosettes were studied in details in many papers, see \cite{CM1, MM1, MM2, Moz1} and the literature therein. In particular well studied ovals are examples of rosettes.

\begin{defn}
Let $m$ be a positive integer. An $m$--\textit{rosette} is a regular closed curve with non--vanishing curvature of the rotation number equal to $m$.
\end{defn}

\begin{rem}
By the above definition we get also that every rosette is positively oriented.
\end{rem}

Let $R_{m}$ be an $m$--rosette. Let $M$ be a regular closed curve.

\begin{defn}\label{parallelpair}
A pair of points $a,b\in M$ ($a\neq b$) is called a \textit{parallel pair} if the tangent lines to $M$ at $a$ and $b$ are parallel.
\end{defn}

\begin{defn}\label{equidistantSet}
An \textit{affine} $\lambda$\textit{--equidistant} is the following set
$$E_{\lambda}(M)=\left\{\lambda a+(1-\lambda)b\ \big|\ a,b \text{ is a parallel pair of } \M\right\}.$$

The set $\Eq_{0.5}(\M)$ will be called the \textit{Wigner caustic} of $\M$.
\end{defn}

Note that, for any given $\lambda\in\mathbb{R}$, we have $E_{\lambda}(M)=E_{1-\lambda}(M)$. Thus, the case $\displaystyle\lambda=\frac{1}{2}$ is special. In particular we have $E_0(M)=E_1(M)=M$.

If $M$ is a generic closed regular curve then $E_{\lambda}(M)$ for $\lambda\neq 0, 1$ is a union of smooth parametrized curves. Each of these curves we will call a \textit{smooth branch} of $E_{\lambda}(M)$. Generically each such branch has at most cusp singularities  (\cite{B1, GZ1}).

In Fig. \ref{PictureBranchesWC} we present (i) a rosette $R_4$, (ii) $E_{0.4}(R_4)$, (iii) $E_{0.5}(R_4)$ and in (iv-viii) different smooth branches of $E_{0.5}(R_4)$.

\begin{figure}[h]
\centering
\includegraphics[scale=0.195]{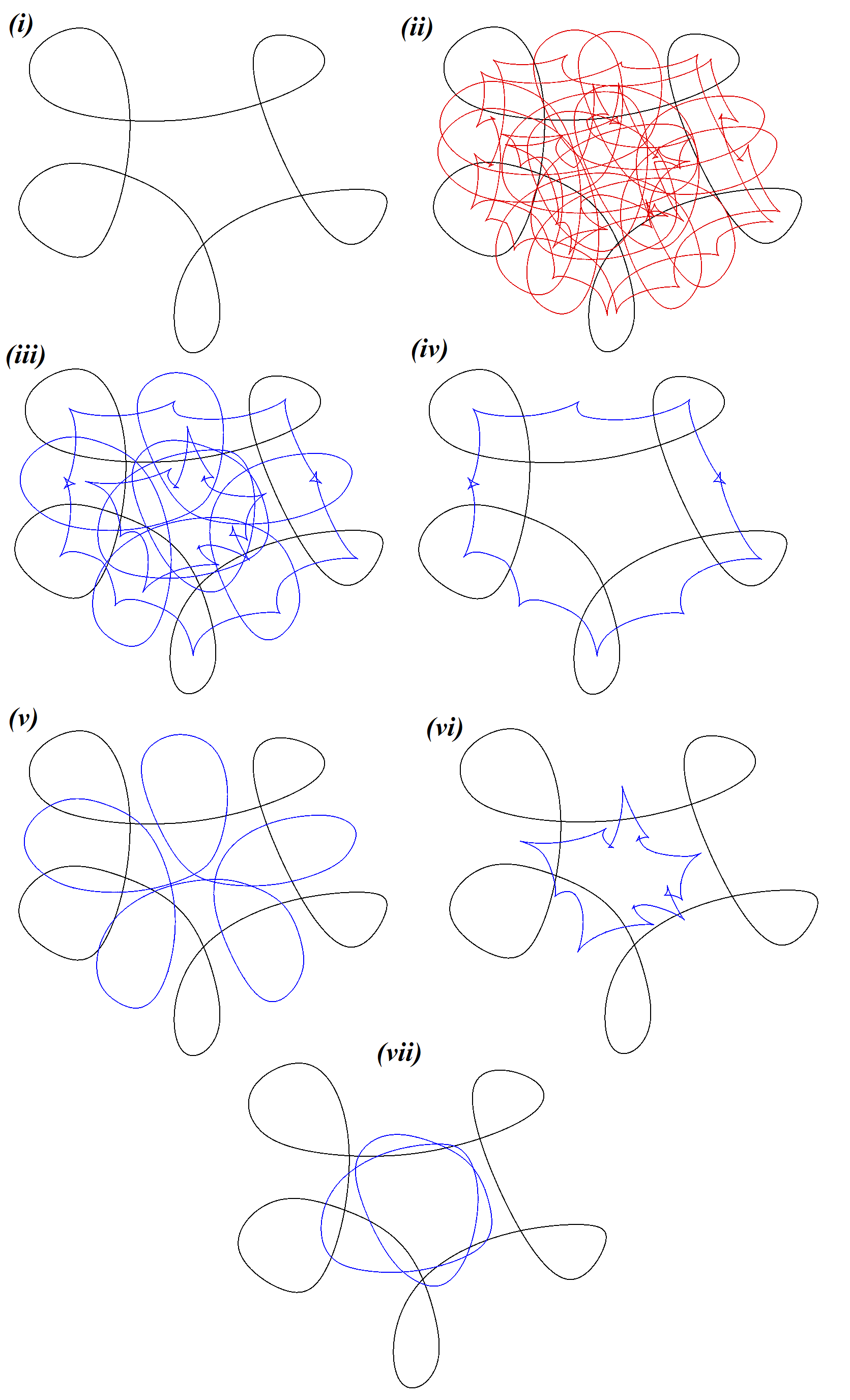}
\caption{}
\label{PictureBranchesWC}
\end{figure}

\begin{defn}
A smooth curve $\mathcal{L}$ is a \textit{convex loop} if $\mathcal{L}$ is parametrized by a smooth function $f:[0,b]\to\mathbb{R}^2$ such that the region bounded by $\mathcal{L}$ is strictly convex, $f(0)=f(b)$ and vectors $f'(0)$, $f'(b)$ are linearly independent.
\end{defn}

In \cite{DZ1} we study the geometry of $E_{\lambda}$ of rosettes and obtain the following result.

\begin{thm}\label{ThmGeometryRosettesDZ}\cite{DZ1}
Let $R_m$ be a generic rosette. Then
\begin{enumerate}[(i)]
\item the number of smooth branches of $E_{0.5}(R_m)$ is equal to $m$,
\item at least $\displaystyle\left\lfloor\frac{m}{2}\right\rfloor$ branches of $E_{0.5}(R_m)$ are rosettes,
\item $m-1$ branches of $E_{0.5}(R_m)$ have a rotation number equal to $m$ and one branch has a rotation number equal to $\displaystyle\frac{m}{2}$,
\item every smooth branch of $E_{0.5}(R_m)$ has an even number of cusps if $m$ is even,
\item exactly one branch of $E_{0.5}(R_m)$ has an odd number of cusps if $m$ is odd,
\item cusps of $E_{0.5}(R_n)$ created from convex loops of $M$ are in the same smooth branch of $E_{0.5}(R_m)$,
\item the total number of cusps of $E_{0.5}(R_m)$ is not smaller than $2$,
\item the number of smooth branches of $E_{\lambda}(R_m)$ for $\displaystyle\lambda\neq 0, 0.5, 1$ is equal to $2m-1$,
\item the rotation number of every smooth branch of $E_{\lambda}(R_m)$ for $\displaystyle\lambda\neq 0, 0.5, 1$ is equal to $m$,
\item at least $m-1$ branches of $E_{\lambda}(R_m)$ for $\displaystyle\lambda\in\left(0,0.5\right)\cup\left(0.5,1\right)$ are rosettes,
\item at least $m$ branches of $E_{\lambda}(R_m)$ for $\lambda\in(-\infty,0)\cup(1,\infty)$ are rosettes.
\end{enumerate}
\end{thm}

In this section we improve Theorem \ref{ThmGeometryRosettesDZ} - see Theorem \ref{ThmLenWCRm}, Theorem \ref{ThmRosettesOfTheWC} and Theorem \ref{ThmIsoEq}.

Let us consider an oval $R_1$. Take a point $\Theta$ as the origin of our frame. Let $p$ be the oriented perpendicular distance from $\Theta$ to the tangent line at a point on $R_1$. Let $\theta$ be the oriented angle from the positive $x_1$--axis to this perpendicular ray. Clearly, $p$ is a $2\pi$--periodic function. It is well known (see for instance \cite{G4, H2, San}) that a parameterization of $R_1$ in terms of $\theta, p(\theta)$ is as follows
\begin{align}\label{OvalParameterization}
\gamma_{R_1}(\theta) &=\big(\gamma_{R_1, 1}(\theta),\gamma_{R_1, 2}(\theta)\big)=\big(p(\theta)\cos\theta-p'(\theta)\sin\theta, p(\theta)\sin\theta+p'(\theta)\cos\theta\big).
\end{align}
The pair $\theta, p(\theta)$ is known as the \textit{polar tangential coordinate} on $R_1$ and $p(\theta)$ its \text{Minkowski's support function}.

Then, the curvature $\kappa$ of $R_1$ is in the following form
\begin{align}\label{CurvatureOval}
\displaystyle \kappa(\theta)=\frac{d\theta}{ds}=\frac{1}{p(\theta)+p''(\theta)}>0,
\end{align}
or equivalently, the radius of a curvature $\rho$ of $R_1$ is given by
\begin{align}\label{RadiusCurvOval}
\rho(\theta)=\frac{ds}{d\theta}=p(\theta)+p''(\theta).
\end{align}
Let $L_{R_1}$ and $A_{R_1}$ be the length of $R_1$ and the area it bounds, respectively. Then one can get that
\begin{align}\label{CauchyFormula}
L_{R_1}=\int_{R_1}ds=\int_0^{2\pi}\rho(\theta)d\theta=\int_0^{2\pi}p(\theta)d\theta,
\end{align}
and
\begin{align}\label{BlaschkeFormula}
A_{R_1} & =\frac{1}{2}\int_{R_1}p(\theta)ds\\
\nonumber	&=\frac{1}{2}\int_0^{2\pi}p(\theta)\left[p(\theta)+p''(\theta)\right]d\theta=\frac{1}{2}\int_0^{2\pi}\left[p^2(\theta)-p'^2(\theta)\right]d\theta.
\end{align}
(\ref{CauchyFormula}) and (\ref{BlaschkeFormula}) are known as \textit{Cauchy's formula} and \textit{Blaschke's formula}, respectively.

The notion of support function can be generalized to obtain function of similar properties, but for rosettes (\cite{MM1, San}), but in this case the support function of $R_m$ is $2m\pi$--periodic.

Let $p$ be the support function of $R_m$ and $\gamma$ be the parameterization of $R_m$.

Let $E_{0.5, k}(R_m)$ for $k=1, 2, \ldots, m$ be different branches of the Wigner caustic of $R_{m}$. Then by \cite{DZ1} the parameterization of $E_{0.5, k}(R_m)$ is as follows.
\begin{align}
\gamma_{0.5, k}(\theta)=\frac{1}{2}\big(\gamma(\theta)+\gamma(\theta+k\pi)\big).
\end{align}
\begin{rem}\label{RemDoubleCov}
The map $R_m\ni\gamma(\theta)\mapsto\gamma_{0.5, m}(\theta)\in E_{0.5, m}(R_m)$ is the double covering of $E_{0.5, m}(R_m)$.
\end{rem}

Then the support function of $E_{0.5, k}(R_m)$ is in the following form.
\begin{align}\label{SupportBranchWC}
p_{\lambda, k}(\theta)=\lambda p(\theta)+(-1)^k (1-\lambda) p(\theta+k\pi).
\end{align}

In Fig. \ref{PictureBranchesWC}(iv) there is $E_{0.5, 1}(R_4)$, in (v) $E_{0.5, 2}(R_4)$, in (vi) $E_{0.5, 3}(R_4)$ and in (vii) $E_{0.5, 4}(R_4)$.

\begin{thm}\label{ThmLenWCRm}
\begin{enumerate}[(i)]
\item If $k$ is even and $k<m$ then
\begin{align*}
L_{E_{0.5, k}(R_m)}=L_{R_m}.
\end{align*}
\item If $k$ is even and $k=m$ then
\begin{align*}
2L_{E_{0.5, k}(R_m)}=L_{R_m}.
\end{align*}
\item If $k$ is odd and $k<m$ then
\begin{align*}
L_{E_{0.5, k}(R_m)}\leqslant L_{R_m}.
\end{align*}
\item If $k$ is odd and $k=m$ then
\begin{align*}
2L_{E_{0.5, k}(R_m)}\leqslant L_{R_m}.
\end{align*}
\item If $k$ is even and $\lambda\in(0,1)-\{0.5\}$, then
\begin{align*}
L_{E_{\lambda, k}(R_m)}=L_{R_m}
\end{align*}
\item If $k$ is odd and $\lambda\in(0,1)-\{0.5\}$, then
\begin{align*}
L_{E_{\lambda, k}(R_m)}\leqslant L_{R_m}
\end{align*}
\item If $k$ is odd and $\lambda\in(-\infty,0)\cup(1,\infty)$, then
\begin{align*}
L_{E_{\lambda, k}(R_m)}=L_{R_m}
\end{align*}
\item If $k$ is even and $\lambda\in(-\infty,0)\cup(1,\infty)$, then
\begin{align*}
L_{E_{\lambda, k}(R_m)}\leqslant (|\lambda|+|1-\lambda|)L_{R_m}
\end{align*}
\end{enumerate}
\end{thm}
\begin{proof}

We prove (i--iv), other points are analogous.

Let us notice that $\rho_{0.5, k}(\theta)=\frac{1}{2}\big|\rho(\theta)+(-1)^k\rho(\theta+k\pi)\big|$.

(i) If $k$ is even and $k<m$, then
\begin{align*}
L_{E_{0.5, k}(R_m)} &=\int_0^{2m\pi}\rho_{0.5, k}(\theta)d\theta\\
	&=\frac{1}{2}\int_0^{2m\pi}(\rho(\theta)+\rho(\theta+k\pi)d\theta\\
	&=\frac{1}{2}(L_{R_m}+L_{R_{m}})=L_{R_m}.
\end{align*}

(ii) It is a consequence of (i) and Remark \ref{RemDoubleCov}.

(iii) If $k$ is odd and $k<m$, then
\begin{align*}
L_{E_{0.5, k}(R_m)} &=\int_0^{2m\pi}\rho_{0.5, k}(\theta)d\theta\\
	&=\frac{1}{2}\int_0^{2m\pi}|\rho(\theta)-\rho(\theta+k\pi|d\theta\\
	&\leqslant \frac{1}{2}(L_{R_m}+L_{R_{m}})=L_{R_m}.
\end{align*}

(iv) It is a consequence of (iii) and Remark \ref{RemDoubleCov}.

\end{proof}

The support function $p$ of $R_{m}$ has the following Fourier series.
\begin{align}\label{FourierP}
p(\theta) = a_0+\sum_{n=1}^{\infty}\left(a_n\cos\frac{n\theta}{m}+b_n\sin\frac{n\theta}{m}\right).
\end{align}
Therefore we get the following equalities.
\begin{align}
p'(\theta) &= \frac{1}{m}\sum_{n=1}^{\infty}n\left(-a_n\sin\frac{n\theta}{m}+b_n\cos\frac{n\theta}{m}\right).\\
p''(\theta) &= -\frac{1}{m^2}\sum_{n=1}^{\infty}n^2\left(a_n\cos\frac{n\theta}{m}+b_n\sin\frac{n\theta}{m}\right).\\
\label{FourierPradius} \rho(\theta) &=p(\theta)+p''(\theta) =a_0-\frac{1}{m^2}\sum_{n=1}^{\infty}(n^2-m^2)\left(a_n\cos\frac{n\theta}{m}+b_n\sin\frac{n\theta}{m}\right).\\
\label{FourierLength} L_{R_m} &=\int_0^{2m\pi}|\gamma'(\theta)|d\theta=\int_0^{2m\pi}\rho(\theta)d\theta=\int_0^{2\pi}p(\theta)d\theta=2\pi ma_0.\\
\label{AreaRFourier} \widetilde{A}_{R_m} &=\frac{1}{2}\int_0^{2m\pi}\Big(p^2(\theta)-p'^2(\theta)\Big)d\theta=m\pi a_0^2-\frac{\pi}{2m}\sum_{n=1}^{\infty}(n^2-m^2)(a_n^2+b_n^2).
\end{align}

\begin{thm}\label{ThmRosettesOfTheWC}
Let $R_m$ be a generic rosette. Then exactly $\displaystyle\left\lfloor\frac{m}{2}\right\rfloor$ branches of $E_{0.5}(R_m)$ are  rosettes.
\end{thm}
\begin{proof}
By Theorem 6.1. in \cite{DZ1} we know that branches $E_{0.5, 2}(R_m)$, $E_{0.5, 4}(R_m)$, $\ldots$, $E_{0.5, 2\cdot\left\lfloor 0.5m\right\rfloor}(R_m)$ are rosettes. We will show that other branches has at least one singular point. It is enough to show that there exists $\theta$ such that $\rho_{0.5, k}(\theta)=0$ for $\displaystyle k=1, 3, \ldots, 2\left\lceil \frac{m}{2}\right\rceil -1$. By (\ref{SupportBranchWC}) we get that
\begin{align*}
\rho_{0.5, k}(\theta) &=0,\\
p_{0.5, k}(\theta)+p''_{0.5,k}(\theta) &=0,\\
\rho(\theta)-\rho(\theta+k\pi) &=0.\\
\end{align*}
So the number of singular points of $E_{0.5, k}(R_m)$ is equal to the number of zeros of $f(\theta)=\rho(\theta)-\rho(\theta+k\pi)$ in the range $[0,2m\pi)$. Since $f(0)=f(2m\pi)$ and $\displaystyle\int_0^{2m\pi}f(\theta)d\theta=0$, therefore there are at least two zeros of $f(\theta)$ in $[0,2m\pi)$.

\end{proof}

The lower bound of $2$ cusps of the branches of the Wigner caustic which are not rosettes can be obtained, see Fig. \ref{FigWC2Cusps}.

\begin{figure}[h]
\centering
\includegraphics[scale=0.3]{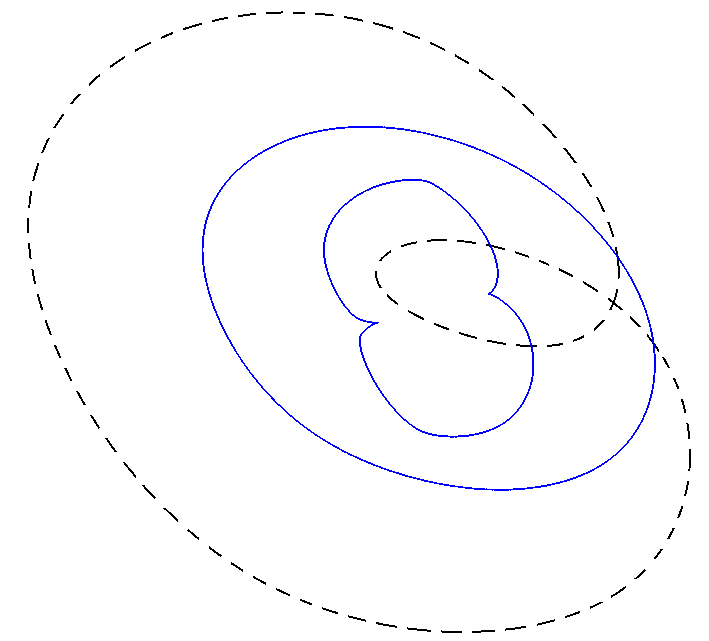}
\caption{A $2$--rosette (the dashed line) and two branches of the Wigner caustic.}
\label{FigWC2Cusps}
\end{figure}

\begin{cor}
A rosette $R_m$ has at least $2\displaystyle\left\lceil\frac{m}{2}\right\rceil$ antipodal pairs.
\end{cor}

\begin{lem}\label{LemFourierAreaBranchWC}
The oriented area of $E_{0.5, k}(R_m)$ for $k<m$ in terms of coefficients of the Fourier series of the Minkowski support function $p$ of $R_m$ is equal
\begin{align}
\label{FourierBranchWC1} \widetilde{A}_{E_{0.5, k}(R_m)} &=\frac{1+(-1)^k}{2}m\pi a_0^2-\frac{\pi}{4m}\sum_{n=1}^{\infty}(n^2-m^2)\left(1+(-1)^k\cos\frac{nk\pi}{m}\right)(a_n^2+b_n^2),
\end{align}
and for $k=m$
\begin{align}
\label{FourierBranchWC2} \widetilde{A}_{E_{0.5, m}(R_m)} &=\frac{1+(-1)^m}{4}m\pi a_0^2-\frac{\pi}{8m}\sum_{n=1}^{\infty}(n^2-m^2)\left(1+(-1)^{n+m}\right)(a_n^2+b_n^2).
\end{align}
\end{lem}
\begin{proof}
By (\ref{SupportBranchWC}) and (\ref{AreaRFourier}) we get the following equalities for $k<m$.
\begin{align*}
\widetilde{A}_{E_{0.5,k}(R_m)} &=\frac{1}{2}\int_0^{2m\pi}\Big(p_{0.5, k}^2(\theta)-p_{0.5, k}'^2(\theta)\Big)d\theta\\
	&=\frac{1}{8}\int_0^{2m\pi}\left[\big(p(\theta)+(-1)^kp(\theta+k\pi)\big)^2-\big(p'(\theta)+(-1)^kp'(\theta+k\pi)\big)\right]d\theta\\
	&=\frac{1}{2}\widetilde{A}_{R_m}+\frac{(-1)^k}{2}\Psi_{R_m},
\end{align*}
where
\begin{align}
\Psi_{R_m}=\frac{1}{2}\int_0^{2m\pi}\Big(p(\theta)p(\theta+k\pi)-p'(\theta)p'(\theta+k\pi)\Big)d\theta.
\end{align}
By (\ref{FourierP}) we get the following:
\begin{align}\label{FourierPkpi}
p(\theta+k\pi)=a_0+\sum_{n=1}^{\infty}\left[a_n\cos\left(\frac{n\theta}{m}+\frac{nk\pi}{m}\right)+b_n\sin\left(\frac{n\theta}{m}+\frac{nk\pi}{m}\right)\right]d\theta.
\end{align}
One can check that
\begin{align}
\label{IntegralFourierA}\int_0^{2m\pi}\cos\frac{n_1\theta}{m}\cos\left(\frac{n_2\theta}{m}+\frac{n_2k\pi}{m}\right)d\theta &= m\pi\cos\frac{n_2k\pi}{m}\cdot\delta_{n_1, n_2},\\
\label{IntegralFourierB}\int_0^{2m\pi}\sin\frac{n_1\theta}{m}\sin\left(\frac{n_2\theta}{m}+\frac{n_2k\pi}{m}\right)d\theta &= m\pi\cos\frac{n_2k\pi}{m}\cdot\delta_{n_1, n_2},\\
\label{IntegralFourierC}\int_0^{2m\pi}\cos\frac{n_1\theta}{m}\sin\left(\frac{n_2\theta}{m}+\frac{n_2k\pi}{m}\right)d\theta &= m\pi\sin\frac{n_2k\pi}{m}\cdot\delta_{n_1, n_2},\\
\label{IntegralFourierD}\int_0^{2m\pi}\sin\frac{n_1\theta}{m}\cos\left(\frac{n_2\theta}{m}+\frac{n_2k\pi}{m}\right)d\theta &= -m\pi\sin\frac{n_2k\pi}{m}\cdot\delta_{n_1, n_2},
\end{align}
where $\delta_{n_1,n_2}$ is the Kronecker delta, i.e. $\delta_{n_1,n_2}=\left\{\begin{array}{l} 1,\ n_1=n_2 \\ 0,\ n_1\neq n_2\end{array}\right.$.

Therefore by (\ref{IntegralFourierA}--\ref{IntegralFourierD}) we can express $\Psi_{R_m}$ in the following way.
\begin{align}\label{PsiFourier}
\Psi_{R_m}&=\frac{1}{2}\int_0^{2m\pi}\Big(p(\theta)p(\theta+k\pi)-p'(\theta)p'(\theta+k\pi)\Big)d\theta\\
\nonumber	&=m\pi a_0^2-\frac{\pi}{2m}\sum_{n=1}^{\infty}(n^2-m^2)\cos\frac{nk\pi}{m}(a_n^2+b_n^2).
\end{align}
By (\ref{AreaRFourier}) and (\ref{PsiFourier}) we get (\ref{FourierBranchWC1}) and by Remark \ref{RemDoubleCov} we get (\ref{FourierBranchWC2})

\end{proof}

\begin{thm}\label{ThmRosWCMM1}( Corollary 5.1 in \cite{MM1})
Each rosette of constant width of the rotation number equal to $m$ is given by the following support function.
\begin{align}\label{SupportRosOfCW}
p_{R_m}(\theta)=a_0+\sum_{\substack{n=1,\\ n\ is\ odd}}^{\infty}\left(a_n\cos\frac{n\theta}{m}+b_n\sin\frac{n\theta}{m}\right).
\end{align}
\end{thm}

\begin{thm}
A rosette $R_m$ is a rosette of constant width if and only if
\begin{align}\label{RmCWcondition}
L_{R_m}^2=8\pi m\widetilde{A}_{E_{0.5, m}(R_m)}+2\pi m\big(1-(-1)^m\big)\widetilde{A}_{R_m}.
\end{align}
\end{thm}
\begin{proof}
By (\ref{SupportRosOfCW}) and by (\ref{AreaRFourier}) we get that the oriented area of $R_m$ is in the following form.
\begin{align*}
\widetilde{A}_{R_m} &=m\pi a_0^2-\frac{\pi}{2m}\sum_{\substack{n=1, \\ n\ is\ odd}}^{\infty}(n^2-m^2)(a_n^2+b_n^2).
\end{align*}
Therefore by Lemma \ref{LemFourierAreaBranchWC} and by (\ref{FourierLength}) we get the following equalities.
\begin{align*}
2\widetilde{A}_{E_{0.5, m}(R_m)} &=\frac{1+(-1)^m}{2}m\pi a_0^2-\frac{\pi}{4m}\sum_{\substack{n=1,\\ n\ is\ odd}}^{\infty}(n^2-m^2)(1-(-1)^m)(a_n^2+b_n^2)\\
 &=\frac{1+(-1)^m}{2}m\pi a_0^2-\frac{1-(-1)^m}{2}\left(\widetilde{A}_{R_m}-m\pi a_0^2\right)\\
 &=\pi ma_0^2-\frac{1-(-1)^m}{2}\widetilde{A}_{R_m}\\
 &=\frac{L_{R_m}^2}{4\pi m}-\frac{1-(-1)^m}{2}\widetilde{A}_{R_m}.
\end{align*}
Therefore we get (\ref{RmCWcondition}).

\end{proof}

In \cite{Zw2} we introduced the following set for $1$--rosettes.

\begin{defn} 
The \textit{Constant Width Measure Set} of $R_1$ is the following set
\begin{align*}
\Cwms(R_1)=\left\{a-b+\frac{L_{R_1}}{\pi}\cdot\mathbbm{n}(a)\ \big|\ a,b\text{ is a parallel pair of } R_1\right\},
\end{align*}
where $\mathbbm{n}$ is the unit normal vector field compatible with the orientation of $R_1$. We treat $\Cwms(R_1)$ as a subset of a vector space $V=\mathbb{R}^2$. To visualize the $\Cwms$ and the other sets in the same picture let us assume that the origin of $V$ is $\Theta$.

\end{defn}

We generalize the $\Cwms$ for all rosettes in the following way.

\begin{defn}
Let $\gamma$ be parameterization of $R_m$ in terms of $\theta$ and $p(\theta)$. Then the \textit{Constant Width Measure Set} of $R_m$ is a curve defined by the following parameterization
\begin{align*}
\gamma_{\Cwms(R_m)}(\theta)=\gamma(\theta)-\gamma(\theta+m\pi)+\frac{L_{R_m}}{m\pi}\cdot\mathbbm{n}(\theta).
\end{align*}

\end{defn}

It is easy to see that $\Cwms(R_m)$ has the following support function

\begin{align}\label{CwmsSupport}
p_{\Cwms(R_m)}(\theta)=p(\theta)-(-1)^mp(\theta+m\pi)-\frac{L_{R_m}}{m\pi}
\end{align}

and in terms of coefficients of the Fourier series of $p$

\begin{align}\label{CwmsSupportFourier}
p_{\Cwms(R_m)}(\theta)=\big(-1-(-1)^m\big)a_0+\sum_{n=1}^{\infty}\big(1-(-1)^{n+m}\big)\left(a_n\cos\frac{n\theta}{m}+b_n\sin\frac{n\theta}{m}\right).
\end{align}

\begin{thm}
Let $R_m$ be a generic $m$-rosette. If $m$ is odd then $\Cwms(R_m)$ is a smooth curve with only cusp singularities and the number of cusp singularities of $\Cwms(R_m)$ is even and positive.
\end{thm}
\begin{proof}
$\Cwms(R_m)$ is a connected smooth curve. The set $\Cwms(R_m)$ is singular if and only if 
\begin{align}\label{SingConditionCwms}
\rho_{\Cwms(R_m)}(\theta)=0.
\end{align}
By theory of Thom (1975) \cite{T1} one can get that generically $\Cwms(\M)$ has cusp singularity when (\ref{SingConditionCwms}) holds.

If $m$ is odd the condition (\ref{SingConditionCwms}) is equivalent to
\begin{align}\label{SingConditionCwmsOdd}
\rho(\theta)+\rho(\theta+m\pi)=\frac{L_{R_m}}{m\pi}.
\end{align}

The number of cusp singularities of $\Cwms(R_m)$ is the number of zeros of (\ref{SingConditionCwmsOdd}) in the range $[0,2m\pi)$. Since the function $\rho(\theta)+\rho(\theta+m\pi)$ is $m\pi$--periodic, we get that the number of cusp singularities of $\Cwms(R_m)$ is even. Let us assume that this number is equal to zero. Then we get that for all $\theta\in[0,m\pi)$ one of two following inequalities holds.
\begin{align*}
\rho(\theta)+\rho(\theta+m\pi)<\frac{L_{R_m}}{m\pi},\\
\rho(\theta)+\rho(\theta+m\pi)>\frac{L_{R_m}}{m\pi}.
\end{align*}
If $\displaystyle\rho(\theta)+\rho(\theta+m\pi)<\frac{L_{R_m}}{m\pi}$ for all $\theta\in[0,m\pi)$, then 
\begin{align*}
\int_0^{m\pi}\left(\rho(\theta)+\rho(\theta+m\pi)\right)d\theta &<\int_0^{m\pi}\frac{L_{R_m}}{m\pi}d\theta,\\
L_{R_m}&<L_{R_m}.
\end{align*}
Since we get the same contradiction for $\displaystyle\rho(\theta)+\rho(\theta+m\pi)>\frac{L_{R_m}}{m\pi}$, the number of cusp singularities of $\Cwms(R_m)$ is positive.

\end{proof}

In Fig. \ref{FigCwmsRm} we present examples of $\Cwms(R_m)$.

\begin{figure}[h]
\centering
\includegraphics[scale=0.17]{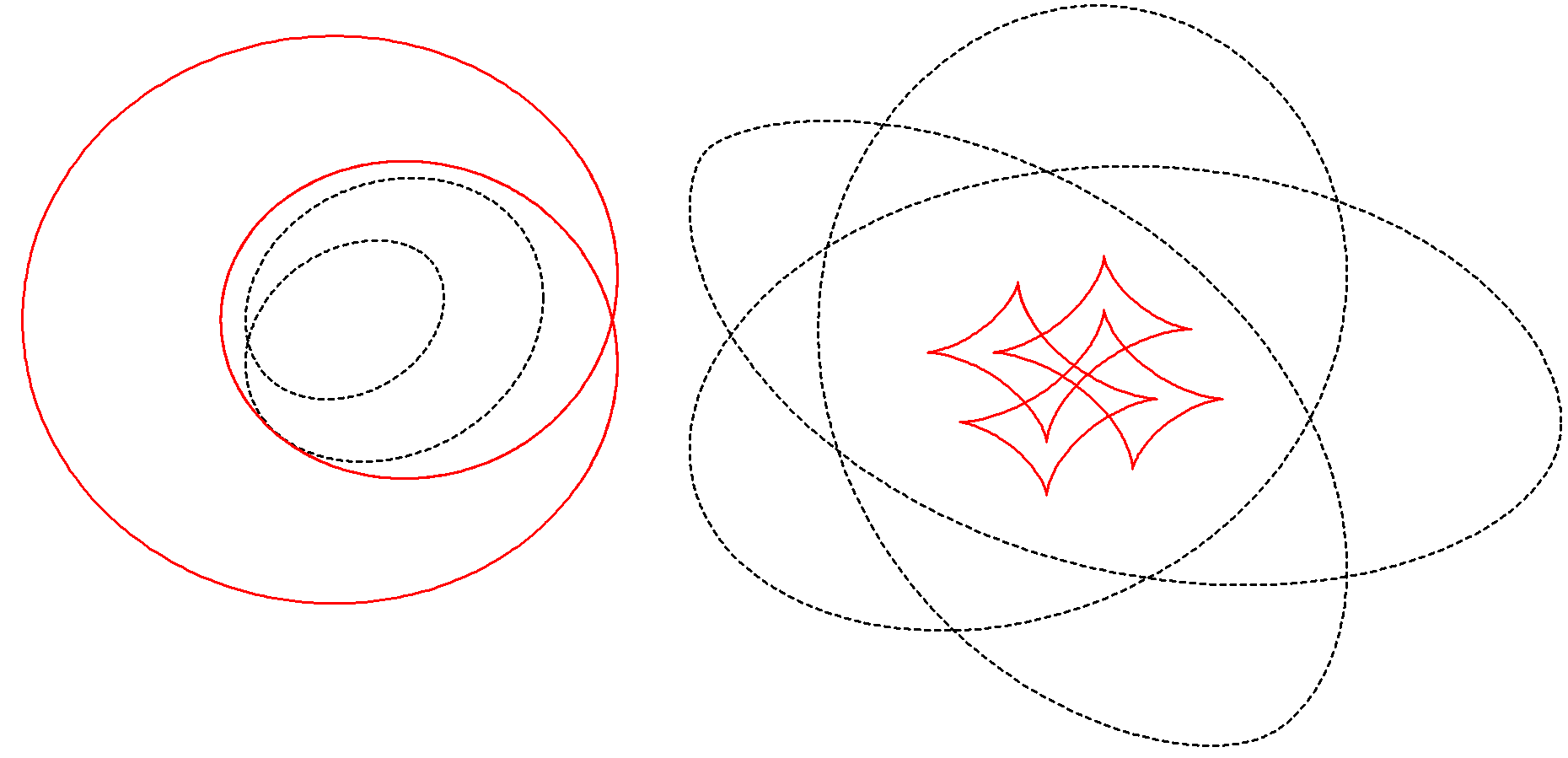}
\caption{Two rosettes $R_2$ and $R_3$ (the dashed lines) and $\Cwms(R_2), \Cwms(R_3)$. Their support functions are $p_{R_2}(\theta)=10+4\cos\frac{\theta}{2}+\sin 2\theta$ and $p_{R_3}(\theta)=12+\cos\frac{2\theta}{3}+4\cos\frac{5\theta}{3}-\sin 2\theta$, respectively.}
\label{FigCwmsRm}
\end{figure}

\begin{thm}\label{ThmIsoEq}(Isoperimetric equalities for rosettes I)
If $m$ is an odd number, then 
\begin{align}\label{IsoEq1}
L^2_{R_m} &= 4\pi m\widetilde{A}_{R_m}-8\pi m\widetilde{A}_{E_{0.5, m}(R_m)}-\pi m\widetilde{A}_{\Cwms(R_m)}.
\end{align}

If $m$ is an even number, then
\begin{align}\label{IsoEq2}
L^2_{R_m} &= -2\pi m\widetilde{A}_{R_m}+4\pi m\widetilde{A}_{E_{0.5, m}(R_m)}+\frac{\pi m}{2}\widetilde{A}_{\Cwms(R_m)}.
\end{align}
\end{thm}
\begin{proof}
Let $m$ be odd. Then by (\ref{SupportBranchWC}) and (\ref{CwmsSupport}) we get that
\begin{align*}
p_{E_{0.5, m}(R_m)}(\theta) &=\sum_{\substack{n=1,\\ n\ is\ odd}}^{\infty}\left(a_n\cos\frac{n\theta}{m}+b_n\sin\frac{n\theta}{m}\right),\\
p_{\Cwms(R_m)}(\theta) &=2\sum_{\substack{n=2,\\ n\ is\ even}}^{\infty}\left(a_n\cos\frac{n\theta}{m}+b_n\sin\frac{n\theta}{m}\right).
\end{align*}

Therefore
\begin{align*}
\widetilde{A}_{E_{0.5, m}(R_m)} &=-\frac{\pi}{4m}\sum_{\substack{1=2,\\ n\ is\ odd}}^{\infty}(n^2-m^2)(a_n^2+b_n^2),\\
\widetilde{A}_{\Cwms(R_m)} & =-\frac{2\pi}{m}\sum_{\substack{n=2,\\ n\ is\ even}}^{\infty}(n^2-m^2)(a_n^2+b_n^2).
\end{align*}

Let $m$ be even. Then similarly one can get that
\begin{align*}
p_{E_{0.5, m}(R_m)}(\theta) &=a_0+\sum_{\substack{n=2,\\ n\ is\ even}}^{\infty}\left(a_n\cos\frac{n\theta}{m}+b_n\sin\frac{n\theta}{m}\right),\\
p_{\Cwms(R_m)}(\theta) &= -2a_0+2\sum_{\substack{n=1,\\ n\ is\ odd}}^{\infty}\left(a_n\cos\frac{n\theta}{m}+b_n\sin\frac{n\theta}{m}\right),\\
\widetilde{A}_{E_{0.5, m}(R_m)}&=\frac{m\pi a_0^2}{2}-\frac{\pi}{4m}\sum_{\substack{n=2,\\ n\ is\ even}}^{\infty}(n^2-m^2)(a_n^2+b_n^2),\\
\widetilde{A}_{\Cwms(R_m)}&=8\pi ma_0^2-\frac{2\pi}{m}\sum_{\substack{n=1,\\ n\ is\ odd}}^{\infty}(n^2-m^2)(a_n^2+b_n^2).
\end{align*}

Then by (\ref{AreaRFourier}) it is easy to verify that (\ref{IsoEq1}) and (\ref{IsoEq2}) hold.

\end{proof}

We recall definition of an offset (this set is also known as a \textit{parallel} set).

\begin{defn}
Let $M$ be a regular positively oriented closed curve. An \textit{$\alpha$-offset} of $\M$ is the following set
\begin{align*}
F_{\alpha}(M)=\big\{a+\alpha\cdot\mathbbm{n}(a)\ \big|\ a\in M\big\},
\end{align*}
where $\mathbbm{n}(a)$ is a continuous unit normal vector field to $M$ at $a$ compatible with the orientation of $M$.
\end{defn}

Offset curves and surfaces are well-known geometric objects in the field of mathematics and computer aided geometric design, possibly because they give a powerful tool in many applications. It is well known that offsets of curves generically admit at most cusp singularities, singularities of all offsets of a curve form an evolute and set of all points of self--intersections of offsets forms a medial axis. For details see \cite{FHK1, GW1, HL1, KGP1, PP1} and the literature therein.

In \cite{Zw2} we define the Spherical Measure Set for simply closed curves as follows.

\begin{defn}
The \textit{Spherical Measure Set} of a regular positively oriented simply closed curve $M$ is an offset at level $\displaystyle \frac{L_{M}}{2\pi}$,
\begin{align}
\Sms(M)=F_{\frac{L_{M}}{2\pi}}(M)=\left\{a+\frac{L_{M}}{2\pi}\cdot\mathbbm{n}(a)\ \Big|\ a\in M\right\},
\end{align}
where $\mathbbm{n}$ is the unit normal vector field compatible with the orientation of $M$.
\end{defn}

In this paper we generalize this notion for rosettes.

\begin{defn}
The \textit{Spherical Measure Set} of an $m$-rosette $R_m$ is an offset at level $\displaystyle \frac{L_{R_m}}{2m\pi}$,
\begin{align}
\Sms(R_m)=F_{\frac{L_{R_m}}{2m\pi}}(R_m)=\left\{a+\frac{L_{R_m}}{2m\pi}\cdot\mathbbm{n}(a)\ \Big|\ a\in R_m\right\}.
\end{align}
\end{defn}

It is easy to see that the support function of $\Sms(R_m)$ is in the following form
\begin{align}\label{SupportSms}
p_{\Sms(R_m)}(\theta)=p(\theta)-\frac{L_{R_m}}{2m\pi},
\end{align}
and has the following Fourier series
\begin{align}\label{SupportSmsFourier}
p_{\Sms(R_m)}(\theta)=\sum_{n=1}^{\infty}\left(a_n\cos\frac{n\pi}{m}+b_n\sin\frac{n\pi}{m}\right).
\end{align}

Directly by (\ref{SupportSmsFourier}) we got the following propositions.

\begin{prop}
A rosette $R_m$ is a circle if and only if $\Sms(R_m)=\{\Theta\}$.
\end{prop}

\begin{prop}
If $m$ is odd and $R_m$ is a rosette of constant width, then $\Sms(R_m)=E_{0.5, m}(R_m)$.
\end{prop}

\begin{thm}
Let $R_m$ be a generic rosette. Then the number of cusp singularities of $\Sms(R_m)$ is even and not smaller than $2$.
\end{thm}
\begin{proof}
A curve $\Sms(R_m)$ is singular if and only if
\begin{align}
\rho_{\Sms(R_m)}(\theta)=0,
\end{align}
which is equivalent to
\begin{align}
\rho(\theta)-\frac{L_{R_m}}{2m\pi}=0.
\end{align}
The number of cusp singularities of $\Sms(R_m)$ is equal to the number of zeros of $\displaystyle f(\theta)=\rho(\theta)-\frac{L_{R_m}}{2m\pi}$ in the range $[0,2m\pi)$. Since $f(0)=f(2m\pi)$, the number of zeros is even. Since $\displaystyle\int_0^{2m\pi}f(\theta)d\theta=L_{R_m}-L_{R_m}=0$, there are at least two zeros of $f(\theta)$ in the range $[0,2m\pi)$.

\end{proof}

In Fig. \ref{FigSmsEx} we present examples of the $\Sms$ of a $2$--rosette and a $3$--rosette.

\begin{figure}[h]
\centering
\includegraphics[scale=0.3]{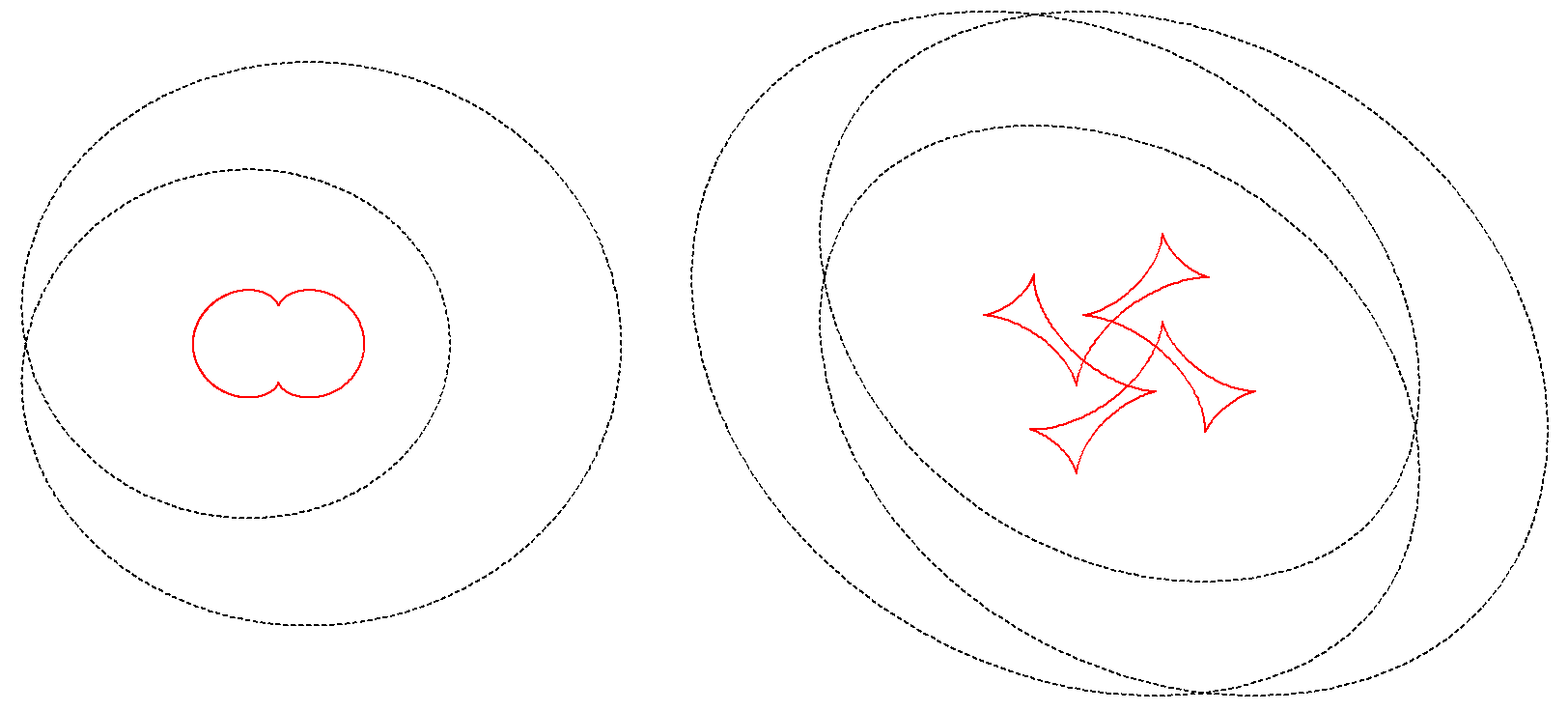}
\caption{Rosettes $R_2$ and $R_3$ (the dashed lines) and $\Sms(R_2), \Sms(R_3)$. Support functions of $R_2$ and $R_3$ are equal to $p_{R_2}(\theta)=3+\cos\frac{\theta}{2}$ and $p_{R_3}(\theta)=4+\cos\frac{2\theta}{3}-\frac{1}{2}\sin 2\theta$, respectively.}
\label{FigSmsEx}
\end{figure}

\begin{thm}\label{ThmIsoEq2}(Isoperimetric equality for rosettes II)
If $m$ is odd and $R_m$ is an $m$-rosette which is not of constant width, then
\begin{align}\label{IsoEq3}
L^2_{R_m} &= 4\pi m\widetilde{A}_{R_m}+4\pi m\widetilde{A}_{\Sms(R_m)}.
\end{align}
\end{thm}
\begin{proof}
By (\ref{SupportSmsFourier}) the oriented area of the Spherical Measure Set of $R_m$ is in the following form
\begin{align}\label{FourierAreaSms}
\widetilde{A}_{\Sms(R_m)}&=\frac{1}{2}\int_0^{2m\pi}\Big(p_{\Sms(R_m)}^2(\theta)-p_{\Sms(R_m)}'^2(\theta)\Big)d\theta\\
\nonumber	&=-\frac{\pi}{2m}\sum_{n=1}^{\infty}(n^2-m^2)(a_n^2+b_n^2).
\end{align}
Then (\ref{IsoEq3}) is an easy consequence of (\ref{FourierLength}), (\ref{AreaRFourier}) and (\ref{FourierAreaSms}).

\end{proof}

%%%%%%%%%%%%%%%%%%%%%%%%%%%%%%%%%%%%%%%%%%%%%%%%%%%%%
%%%%%%%%%%%%%%%%%%%%%%%%%%%%%%%%%%%%%%%%%%%%%%%%%%%%%
%%%%%%%%%%%%%%%%%%%%%%%%%%%%%%%%%%%%%%%%%%%%%%%%%%%%%

\section{Affine $\lambda$--equidistants of the union of rosettes}\label{SecWCUnion}

We recall the definitions and theorems from \cite{DZ1} and then prove Theorem \ref{ThmGeometryEq2Rosettes} and Theorem \ref{ThmLenEq2Rosettes}. We also simplify definitions and theorems from \cite{DZ1} to match only rosettes.

\begin{defn}
Let $R_m$ be a generic $m$--rosette and let $f:S^1\to\mathbb{R}^2$ be a parameterization of $R_m$ in terms of $p(\theta),\theta$. Then the \textit{sequence of parallel points} $\mathcal{S}$ is the following sequence $(s_0, s_1, \ldots, s_{2m-1})$, where $s_k=k\pi$.
\end{defn}

\begin{defn}
The \textit{set of parallel arcs} for an $m$--rosette is the following set
\begin{align*}
\Phi_0=\big\{(0,1), (1,2), \ldots, (2m-2, 2m-1), (2m-1, 0)\big\}.
\end{align*}
\end{defn}

\begin{prop}(\cite{DZ1})
Let $R_m$ be a generic $m$--rosette and let $f:S^1\to\mathbb{R}^2$ be a parameterization of $R_m$ in terms of $p(\theta),\theta$. For every two tuples $(k_1,l_1), (k_2,l_2)$ in $\Phi_0$ the well defined map
\begin{align*}
f\Big((s_{k_1}, s_{l_1})\Big)\ni p\mapsto P(p)\in f\Big((s_{k_2}, s_{l_2})\Big),
\end{align*}
such that $p,P(p)$ is a parallel pair of $R_m$, is a diffeomorphism.
\end{prop}

\begin{rem}
The set of parallel arcs includes tuples of integers that are related to arcs of the $m$--rosette from which we can construct arcs of branches of an affine equidistants, see Fig. \ref{FigArcsEqParallelArcs}.
\end{rem}

\begin{figure}[h]
\centering
\includegraphics[scale=0.3]{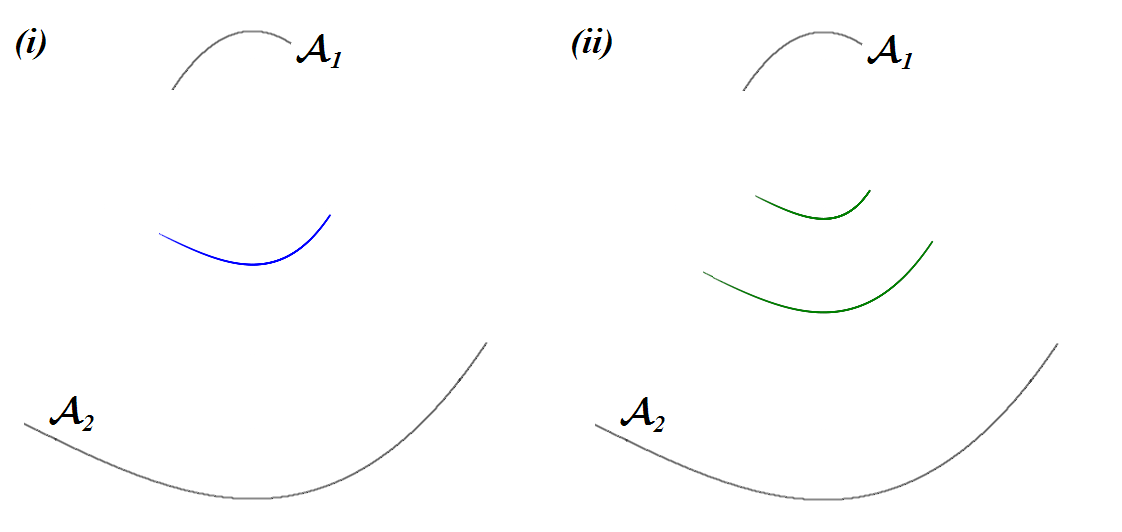}
\caption{Two arcs $\mathcal{A}_1, \mathcal{A}_2$ which are related to the tuples in the set of parallel arcs and \textit{(i)} $E_{\frac{1}{2}}(\mathcal{A}_1\cup\mathcal{A}_2)$, \textit(ii) $E_{\frac{2}{5}}(\mathcal{A}_1\cup\mathcal{A}_2)$.}
\label{FigArcsEqParallelArcs}
\end{figure}

\begin{defn}
Let $(k_1,k_2), (l_1,l_2)$ belong to the set of parallel arcs, then \linebreak $\begin{array}{ccc} \hline k_1 &  & k_2 \\ \hline l_1 &  & l_2 \\ \hline \end{array}$ denotes the following set (an \textit{arc})
\begin{align*}
\Big\{\lambda a+(1-\lambda)b \ \Big|\ a\in f\big( [s_{k_1}, s_{k_2}]\big), b\in f\big([s_{l_1}, s_{l_2}]\big), a,b\text{ is a parallel pair of }R_m\Big\}
\end{align*}
for an affine $\lambda$--equidistant for a fixed value of $\lambda$.

In addition $\begin{array}{ccccc} \hline k_1 & k_2 &\ldots& k_{n-1} & k_n \\ \hline l_1& l_2 &\ldots& l_{n-1} &l_n\\ \hline\end{array}$ denotes $\displaystyle\bigcup_{i=1}^{n-1}\begin{array}{ccc} \hline k_i & &k_{i+1}\\ \hline l_i & &l_{i+1} \\ \hline \end{array}$. We will call this set a \textit{glueing scheme} for $E_{\lambda}(R_m)$ for a fixed value of $\lambda$.
\end{defn}

\begin{defn}
A \textit{maximal glueing scheme} for $E_{\lambda}(R_m)$ for $\lambda\neq 0, 1$ is a glueing scheme which is a maximal element of the set of all glueing schemes for $E_{\lambda}(R_m)$ equipped with the inclusion relation.
\end{defn}

\begin{rem} In \cite{DZ1} we prove among other things that each maximal glueing scheme corresponds to a branch of an affine equidistant. If $\displaystyle\lambda\neq 0,\frac{1}{2}, 1$, then $E_{\lambda}(R_m)$ is the union of $\displaystyle 2{|\Phi_0|\choose 2}$ different arcs. If $\displaystyle\lambda=\frac{1}{2}$, then $E_{\frac{1}{2}}(R_m)$ is the union of $\displaystyle{|\Phi_0|\choose 2}$ different arcs. Furthermore if $\displaystyle\lambda\neq 0, \frac{1}{2}, 1$, then the maximal glueing scheme of a branch of $E_{\lambda}(R_m)$ is in the form $\begin{array}{ccccc} \hline k_1 & k_2 &\ldots& k_{n-1} &k_n \\ \hline l_1& l_2 &\ldots& l_{n-1} & l_n\\ \hline\end{array}$, where $(k_1,l_1)=(k_n,l_n)$ and if $\displaystyle\lambda=\frac{1}{2}$, then the maximal glueing scheme of a branch of $E_{\lambda}(R_m)$ is in the form $\begin{array}{ccccc} \hline k_1 & k_2 &\ldots& k_{n-1} & k_n \\ \hline l_1& l_2 &\ldots& l_{n-1} & l_n\\ \hline\end{array}$, where $(k_1,l_1)=(k_n,l_n)$ or $(k_1,l_1)=(l_n,k_n)$.
\end{rem}

\begin{defn}
Let $M$ be a smooth regular curve. Let $a,b$ be parallel pair of $M$ and let us assume that the curvature at $a$ and $b$ does not vanish. We say that $M$ is \textit{curved in the same side} at $a$ and $b$ if the center of curvature
of $M$ at $a$ and the center of curvature of $\tau_{a-b}(M)$ at $a=\tau_{a-b}(b)$ are lying on the
same side of the tangent line to $M$ at $a$.
\end{defn}

\begin{prop}\label{PropRegularPointsOfEq}(\cite{DZ1})
Let $M$ be a smooth regular curve.
\begin{enumerate}[(i)]
\item If $M$ is curved in the same side side at a parallel pair $a,b$, then $\lambda a+(1-\lambda)b$ is a $C^{\infty}$ regular point of $E_{\lambda}(M)$ for $\lambda\in(0,1)$.
\item If $M$ is not curved in the same side at a parallel pair $a,b$, then $\lambda a+(1-\lambda)b$ is a $C^{\infty}$ regular point of $E_{\lambda}(M)$ for $\lambda\in(-\infty,0)\cup(1,\infty)$.
\end{enumerate}
\end{prop}

Theorem 6.1 in \cite{DZ1} and theorems in previous chapter (mainly Theorems \ref{ThmLenWCRm}, \ref{ThmRosettesOfTheWC} and \ref{ThmIsoEq}) show the geometry of the affine $\lambda$--equdisitants and in parcticular of the Wigner caustic of rosettes in a fairly detailed way. To study the geometry of affine $\lambda$--equidistants for the union of two rosettes we want in particular study the geometry of branches which are created from both $M$ and $N$. Therefore we present the following definition. 

\begin{defn}
Let $M$ and $N$ be planar regular closed curves. Then an affine $\lambda$--equidistant of the pair $M, N$ is the following set
\begin{align*}
E_{\lambda}(M,N)=\Big\{\lambda x+(1-\lambda)y\ \Big| &(x\in M\wedge y\in N)\ \vee\ (x\in N\wedge y\in M), \\
	& x,y\text{ is a parallel pair of }M\cup N\Big\}.
\end{align*}
\end{defn}
Note that $E_{\lambda}(M,M)=M\cup E_{\lambda}(M)$ and $E_{\lambda}(M\cup N)=E_{\lambda}(M)\cup E_{\lambda}(M,N)\cup E_{\lambda}(N)$.

\begin{lem}\label{LemEvenNumberOfCusps}
Let $M$ be a closed smooth curve with at most cusp singularities and let the rotation number of $M$ be an integer. Then the number of cusp singularities of $M$ is even.
\end{lem}
\begin{proof}

\begin{figure}[h]
\centering
\includegraphics[scale=0.4]{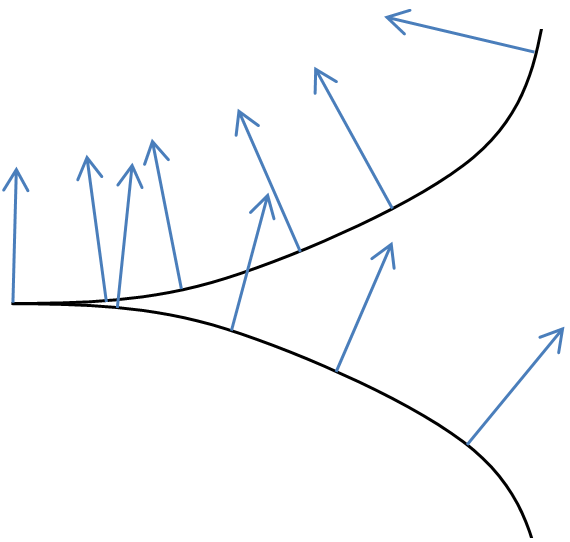}
\caption{The cusp singularity with a continuous normal vector field. Vectors in upper regular component of a curve are directed
outside the cusp, others are directed inside the cusp.}
\label{FigCuspContVectorField}
\end{figure}

A continuous normal vector field to the germ of a curve with the cusp singularity is directed outside the cusp on the one of two connected regular components and is directed inside the cusp on the other component as it is showed in Fig. \ref{FigCuspContVectorField}.

Since the rotation number of $M$ in an integer, we get that the number of cusps $M$ is even.

\end{proof}

Let $\gcd(m,n)$ and $\lcm(m,n)$ denote the greatest common divisor and the least common multiplier of two positive integers $m$ and $n$, respectively.

\begin{thm}\label{ThmGeometryEq2Rosettes}
Let $R_{m_1}$ and $R_{m_2}$ be two generic rosettes. Then
\begin{enumerate}[(i)]
	\item there exist $2\gcd(m_1,m_2)$ branches of $E_{0.5}(R_{m_1}, R_{m_2})$.
	\item at least $\gcd(m_1,m_2)$ branches of $E_{0.5}(R_{m_1},R_{m_2})$ are rosettes.
	\item there exist $4\gcd(m_1,m_2)$ branches of $E_{\lambda}(R_{m_1}, R_{m_2})$ for $\displaystyle\lambda\neq 0,\frac{1}{2}, 1$.
	\item at least $2\gcd(m_1,m_2)$ branches of $E_{\lambda}(R_{m_1},R_{m_2})$ for $\displaystyle\lambda\neq 0,\frac{1}{2}, 1$ are rosettes.
	\item every branch of $E_{\lambda}(R_{m_1},R_{m_2})$ for $\lambda\neq 0, 1$ has the rotation number equal to $\lcm(m_1,m_2)$.
	\item the number of cusps of every branch of $E_{\lambda}(R_{m_1},R_{m_2})$ for $\lambda\neq 0, 1$ is even.
\end{enumerate} 
\end{thm}
\begin{proof}
The sets of parallel arcs $\Phi, \Phi'$ of rosettes $R_{m_1}, R_{m_2}$ are as follows.
\begin{align*}
\Phi &=\big\{(0,1), (1,2), \ldots, (2m_1-2,2m_1-1), (2m_1-1,0)\big\}\\
\Phi' &=\big\{(0',1'), (1',2'), \ldots, (2m_2'-2,2m_2'-1), (2m_2'-1,0')\big\},
\end{align*}
where $a':=a+2m_1$.

Therefore the set of parallel arcs of $R_{m_1}\cup R_{m_2}$ is in the following form.
\begin{align}
\Phi_0=\big\{&(0,1), (1,2), \ldots, (2m_1-2,2m_1-1), (2m_1-1,0), \\
\nonumber	&(0',1'), (1',2'), \ldots, (2m_2'-2,2m_2'-1), (2m_2'-1,0')\big\}.
\end{align}

Let $E_{\lambda, k}(R_{m_1},R_{m_2})$ for an integer $k$ denote a branch of $E_{\lambda}(R_{m_1},R_{m_2})$.

Let us notice that the maximal glueing scheme of $E_{0.5, k}(R_{m_1}, R_{m_2})$ for $k=0, 1, 2, \ldots, 2\gcd(m_1,m_2)-1$ is in the following form
\begin{align}\label{MaxGlueSchemeWC}
\begin{array}{|ccccc|} \hline
0 & 1 & \ldots & 2m_1-1 & 0 \\ \hline
k' & k'+1 & \ldots &k'-1 & k' \\ \hline
\end{array},
\end{align}
where on the upper part of (\ref{MaxGlueSchemeWC}) we have repetitively the following segment 
\begin{align*}
\begin{array}{cccccc}\hline 0 & 1& 2& \ldots & 2m_1-1 & 0\\ \hline\end{array}
\end{align*}
 of the length $2m_1$ and on the lower part of (\ref{MaxGlueSchemeWC}) we have repetitively the following segment \begin{align*}
\begin{array}{cccccc}\hline k' & k'+1& k'+2& \ldots & k'-1 & k'\\ \hline\end{array}
\end{align*} of the length $2m_2$. Therefore the maximal glueing scheme (\ref{MaxGlueSchemeWC}) is made of \linebreak $2\lcm(m_1,m_2)$ arcs.

The number of arcs of branches of $E_{0.5}(R_{m_1}\cup R_{m_2})$ is equal to $\displaystyle {|\Phi_0|\choose 2}$. Because the number of arcs of branches of $E_{0.5}(R_{m_1})$ (resp. of $E_{0.5}(R_{m_2})$) is equal to $\displaystyle {2m_1\choose 2}$ (resp. $\displaystyle {2m_2\choose 2}$), therefore the number of arcs of $E_{0.5}(R_{m_1},R_{m_2})$ is equal to $\displaystyle {|\Phi_0|\choose 2}-{2m_1\choose 2}-{2m_2\choose 2}=4m_1m_2$. Since we already used $2\gcd(m_1,m_2)\cdot 2\lcm(m_1,m_2)=4m_1m_2$ arcs in maximal glueing schemes in the form (\ref{MaxGlueSchemeWC}) we cannot construct more branches of $E_{0.5}(R_{m_1}, R_{m_2})$.

Let $\displaystyle\lambda\neq 0, \frac{1}{2}, 1$. Using the same arguments as previous we can show that all maximal glueing schemes of $E_{\lambda, k}(R_{m_1}, R_{m_2})$ for $k=0, 1, 2, \ldots, 2\gcd(m_1,m_2)-1$ are in the form (\ref{MaxGlueSchemeEq1}) and for $k=2\gcd(m_1,m_2), 2\gcd(m_1,m_2)+1, \ldots, 4\gcd(m_1,m_2)-1$ are in the form (\ref{MaxGlueSchemeEq2}), where $a'':=(a-2\gcd(m_1,m_2))'$.

\begin{align}\label{MaxGlueSchemeEq1}
\begin{array}{|ccccc|} \hline
0 & 1 & \ldots & 2m_1-1 & 0 \\ \hline
k' & k'+1 & \ldots &k'-1 & k' \\ \hline
\end{array}, \\ 
\label{MaxGlueSchemeEq2}
\begin{array}{|ccccc|} \hline
k'' & k''+1 & \ldots &k''-1 & k'' \\ \hline
0 & 1 & \ldots & 2m_1-1 & 0 \\ \hline
\end{array},
\end{align}
where on the upper part of (\ref{MaxGlueSchemeEq1}) (resp. the lower part of (\ref{MaxGlueSchemeEq2}))  we have repetitively the following segment 
\begin{align*}
\begin{array}{cccccc}\hline 0 & 1& 2& \ldots & 2m_1-1 & 0\\ \hline\end{array}
\end{align*}
 of the length $2m_1$ and on the lower part of (\ref{MaxGlueSchemeEq1}) (resp. the upper part of (\ref{MaxGlueSchemeEq2})) we have repetitively the following segment \begin{align*}
&\begin{array}{cccccc}\hline k' & k'+1& k'+2& \ldots & k'-1 & k'\\ \hline\end{array}
\\
(\text{resp. }&\begin{array}{cccccc}\hline k'' & k''+1& k''+2& \ldots & k''-1 & k''\\ \hline\end{array})
\end{align*} of the length $2m_2$.

By (\ref{MaxGlueSchemeWC}), (\ref{MaxGlueSchemeEq1}) and (\ref{MaxGlueSchemeEq2}) let us notice that every branch $E_{\lambda, k}(R_{m_1}, R_{m_2})$ is made of $2\lcm (m_1,m_2)$ arcs and each arc has the rotation number equal to $\displaystyle\frac{1}{2}$. Since the rotation number of every branch $E_{\lambda, k}(R_{m_1}, R_{m_2})$ in an integer and for generic rosettes $R_{m_1}, R_{m_2}$ every branch $E_{\lambda, k}(R_{m_1}, R_{m_2})$ are smooth curves with at most cusp singularities, we get that the number of cusp singularities of every branch $E_{\lambda, k}(R_{m_1}, R_{m_2})$ is even (see Lemma \ref{LemEvenNumberOfCusps}).

Let $S^1\ni s\mapsto f_i(s)\in\mathbb{R}$ be the parameterization of $R_{m_i}$ for $i=1, 2$. Without loss of generality let us assume that curve $R_{m_1}$ at $f_1(0)$ and $R_{m_2}$ at $f_2(0')$ are curved in the same direction. Then one can see that branches $E_{0.5, 2l}(R_{m_1},R_{m_2})$ for $l=0, 1, \ldots, \gcd(m_1,m_2)-1$ have the property that for any parallel pair $a_1,a_2$ such that $a_1\in R_{m_1}, a_2\in R_{m_2}$ and the point $\displaystyle\frac{a_1+a_2}{2}$ belongs to one of the branches above, the curves $R_{m_1}, R_{m_2}$ are curved in the same side, then by Proposition \ref{PropRegularPointsOfEq} these branches are rosettes. Using the same method one can show that at least $2\gcd(m_1,m_2)$ branches of $E_{\lambda}(R_{m_1}, R_{m_2})$ for $\displaystyle\lambda\neq 0, \frac{1}{2}, 1$ are rosettes.

\end{proof}

\begin{rem}
The set $E_{\lambda}(R_{m_1}, R_{m_2})$ can be the union of only rosettes (see Fig. \ref{FigEqOnlyRosettes}).
\end{rem}

\begin{figure}[h]
\centering
\includegraphics[scale=0.24]{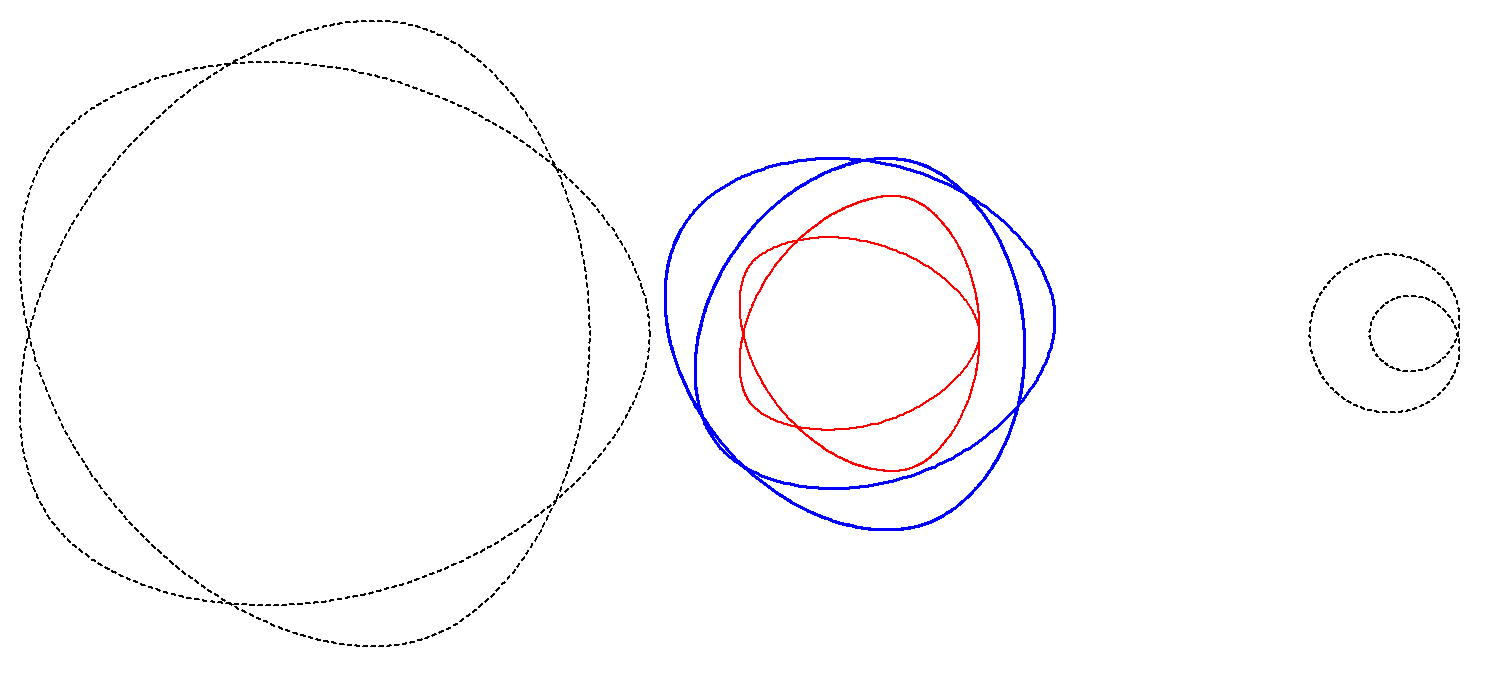}
\caption{Two rosettes $R_2, R_2'$ (the dashed lines) and two branches of $E_{0.5}(R_2, R_2')$ (the normal and the bold line). The support functions of $R_2$ (resp. $R_2'$) is $p_{R_2}(\theta)=10+\cos\frac{5\theta}{2}$ (resp. $p_{R_2'}(\theta)=2+36\cos\theta+\sin\frac{\theta}{2}$).}
\label{FigEqOnlyRosettes}
\end{figure}

By the maximal glueing schemes (\ref{MaxGlueSchemeWC}), (\ref{MaxGlueSchemeEq1}) and (\ref{MaxGlueSchemeEq2}) we can get the following lemma.

\begin{lem}\label{LemSupportBranchEq}
If $\displaystyle\lambda=\frac{1}{2}$ then the $2\pi\lcm(m_1,m_2)$--periodic support function of a branch $E_{0.5, k}(R_{m_1}, R_{m_2})$ for $k=0, 1, 2, \ldots, 2\gcd(m_1,m_2)-1$ is in the following form
\begin{align}
p_{E_{0.5, k}(R_{m_1}, R_{m_2})}(\theta)=\frac{1}{2}\Big(p_{R_{m_1}}(\theta)+(-1)^kp_{R_{m_2}}(\theta+k\pi)\Big).
\end{align}

If $\displaystyle\lambda\neq 0, \frac{1}{2}, 1$ then the $2\pi\lcm(m_1,m_2)$--periodic support function of a branch $E_{\lambda, k}(R_{m_1}, R_{m_2})$ for $k=0, 1, 2, \ldots, 2\gcd(m_1, m_2)-1$ (resp. for $k=2\gcd(m_1,m_2),$ $ 2\gcd(m_1,m_2)+1, \ldots,$ $4\gcd(m_1,m_2)-1$) is in the form (\ref{SupportEq1}) (resp. in the form (\ref{SupportEq2})).
\begin{align}\label{SupportEq1}
p_{E_{\lambda, k}(R_{m_1}, R_{m_2})}(\theta) &=\lambda p_{R_{m_1}}(\theta)+(-1)^k(1-\lambda) p_{R_{m_2}}(\theta+k\pi),\\
\label{SupportEq2} p_{E_{\lambda, k}(R_{m_1}, R_{m_2})}(\theta) &=(1-\lambda)p_{R_{m_1}}(\theta)+(-1)^k\lambda p_{R_{m_2}}\Big(\theta+(k-2\gcd(m_1,m_2))\pi\Big).
\end{align}

\end{lem}

\begin{thm}\label{ThmLenEq2Rosettes}
\begin{enumerate}[(i)]
\item If $k$ is even, $k<2\gcd(m_1,m_2)$ and $\lambda\in(0,1)$, then
\begin{align*}
L_{E_{\lambda, k}(R_{m_1},R_{m_2})}=\lambda \frac{\lcm(m_1,m_2)}{m_1}L_{R_{m_1}}+(1-\lambda)\frac{\lcm(m_1,m_2)}{m_2}L_{R_{m_2}}.
\end{align*}
\item If $k$ is even, $k\geqslant 2\gcd(m_1,m_2)$ and $\displaystyle\lambda\in(0,1)-\left\{\frac{1}{2}\right\}$, then
\begin{align*}
L_{E_{\lambda, k}(R_{m_1},R_{m_2})}=(1-\lambda) \frac{\lcm(m_1,m_2)}{m_1}L_{R_{m_1}}+\lambda \frac{\lcm(m_1,m_2)}{m_2}L_{R_{m_2}}.
\end{align*}
\item If $k$ is odd, $k<2\gcd(m_1,m_2)$ and $\lambda\in(0,1)$, then
\begin{align*}
L_{E_{\lambda, k}(R_{m_1},R_{m_2})}\leqslant \lambda \frac{\lcm(m_1,m_2)}{m_1}L_{R_{m_1}}+(1-\lambda)\frac{\lcm(m_1,m_2)}{m_2}L_{R_{m_2}}.
\end{align*}
\item If $k$ is odd, $k\geqslant 2\gcd(m_1,m_2)$ and $\displaystyle\lambda\in(0,1)-\left\{\frac{1}{2}\right\}$, then
\begin{align*}
L_{E_{\lambda, k}(R_{m_1},R_{m_2})}\leqslant (1-\lambda) \frac{\lcm(m_1,m_2)}{m_1}L_{R_{m_1}}+\lambda \frac{\lcm(m_1,m_2)}{m_2}L_{R_{m_2}}.
\end{align*}
\item If $k$ is odd, $k<2\gcd(m_1,m_2)$ and $\lambda\in(-\infty,0)\cup(1,\infty)$, then
\begin{align*}
L_{E_{\lambda, k}(R_{m_1},R_{m_2})}= |\lambda|\frac{\lcm(m_1,m_2)}{m_1} L_{R_{m_1}}+|1-\lambda|\frac{\lcm(m_1,m_2)}{m_2}L_{R_{m_2}}.
\end{align*}
\item If $k$ is odd, $k\geqslant 2\gcd(m_1,m_2)$ and $\lambda\in(-\infty,0)\cup(1,\infty)$, then
\begin{align*}
L_{E_{\lambda, k}(R_{m_1},R_{m_2})}= |1-\lambda|\frac{\lcm(m_1,m_2)}{m_1} L_{R_{m_1}}+|\lambda|\frac{\lcm(m_1,m_2)}{m_2} L_{R_{m_2}}.
\end{align*}
\item If $k$ is even, $k<2\gcd(m_1,m_2)$ and $\lambda\in(-\infty,0)\cup(1,\infty)$, then
\begin{align*}
L_{E_{\lambda, k}(R_{m_1},R_{m_2})}\leqslant |\lambda|\frac{\lcm(m_1,m_2)}{m_1} L_{R_{m_1}}+|1-\lambda|\frac{\lcm(m_1,m_2)}{m_2}L_{R_{m_2}}.
\end{align*}
\item If $k$ is even, $k\geqslant 2\gcd(m_1,m_2)$ and $\lambda\in(-\infty,0)\cup(1,\infty)$, then
\begin{align*}
L_{E_{\lambda, k}(R_{m_1},R_{m_2})}\leqslant |1-\lambda|\frac{\lcm(m_1,m_2)}{m_1} L_{R_{m_1}}+|\lambda| \frac{\lcm(m_1,m_2)}{m_2}L_{R_{m_2}}.
\end{align*}
\end{enumerate}
\end{thm}
\begin{proof}
We prove (i), other points are analogous.
So let $k$ be an even number and $k<2\gcd(m_1,m_2)$ and let $\lambda\in(0,1)$.
Therefore we get
\begin{align*}
L_{E_{\lambda, k}(R_{m_1}, R_{m_2})} &=\int_0^{2\pi\lcm(m_1,m_2)}\rho_{E_{\lambda, k}(R_{m_1}, R_{m_2})}(
\theta)d\theta\\
	&=\int_0^{2\pi\lcm(m_1,m_2)}p_{E_{\lambda, k}(R_{m_1}, R_{m_2})}(
\theta)d\theta\\
	&=\lambda\int_0^{2\pi\lcm(m_1,m_2)}p_{R_{m_1}}(\theta)d\theta+(1-\lambda)\int_0^{2\pi\lcm(m_1,m_2)}p_{R_{m_2}}(\theta)d\theta\\
	&=\lambda\frac{\lcm(m_1,m_2)}{m_1}L_{R_{m_1}}+(1-\lambda)\frac{\lcm(m_1,m_2)}{m_2}L_{R_{m_2}}.
\end{align*}

\end{proof}

%%%%%%%%%%%%%%%%%%%%%%%%%%%%%%%%%%%%%%%%%

\section*{Acknowledgements}
The work of M. Zwierzy\'nski was partially supported by NCN grant no. DEC-2013/11/B/ST1/03080.
The author is sincerely grateful to Professor Wojciech Domitrz for the valuable discussions.

\bibliographystyle{amsalpha}

\begin{thebibliography}{AAAA}



\bibitem [1] {B1} Berry, M.V.: \emph{Semi-classical mechanics in phase space: a study of Wigner's function}, Philos. Trans. R. Soc. Lond. A 287, 237-271 (1977).

\bibitem [2] {C1} Chavel, I.: \emph{Isoperimetric Inequalities}. Differential Geometric and Analytic Perspectives. Cambridge University Press, 2001.

\bibitem [3] {CM1} Cie\'slak, W., Mozgawa, W.: \emph{On rosettes and almost rosettes}, Geom. Dedicata 24, 221--228 (1987)

\bibitem [4] {C2} Craizer, M.: \emph{Iterates of Involutes of Constant Width Curves in the Mikowski Plane}, Beitr\"age zur Algebra und Geometrie, 55(2), 479--496, 2014.


\bibitem [5] {CDR1} Craizer, M., Domitrz, W., Rios, P. de M.: \emph{Even Dimensional Improper Affine Spheres} , J. Math. Anal. Appl. 421 (2015), pp. 1803--1826.

\bibitem [6] {DMR1} Domitrz, W., Manoel, M., Rios, P. de M.: \emph{The Wigner caustic on shell and singularities of odd functions} , Journal of Geometry and Physics 71(2013), pp. 58-72

\bibitem [7] {DR1} Domitrz, W., Rios, P. de M.: \emph{ Singularities of equidistants and Global Centre Symmetry sets of Lagrangian submanifolds}, Geom. Dedicata 169 (2014), pp. 361-382.

\bibitem [8] {DRS1} Domitrz, W., Rios, P. de M., Ruas, M. A. S.: \emph{Singularities of affine equidistants: projections and contacts}, J. Singul. 10 (2014), 67-81

\bibitem [9] {DZ1} Domitrz, W., Zwierzy\'nski, M.: \emph{The geometry of the Wigner caustic and affine equidistants of planar curves}, arXiv:1605.05361

\bibitem [10] {FHK1} Farin, G., Hoschek, J., Kim, M.-S.: \emph{Handbook of Computer Aided Geometric Design}. North-Holland

\bibitem [11] {FKN1} Ferone, V., Kawohl, B., Nitsch, C.: \emph{The elastica problem under area constraint}, Mathematische Annalen, DOI 10.1007/s00208-015-1284-y

\bibitem [12] {G6} Gage, M. E.: \emph{Curve shortening makes convex curves circular}, Invent. Math., 76(1984), 357--364.

\bibitem [13] {G7} Gage, M.E.: \emph{An isoperimetric inequality with applications to curve shortening}, Duke Math. J. 50 (1983), no. 4, 1225--1229.

\bibitem [14] {G3} Giblin, P.J.: \emph{Affinely invariant symmetry sets}. Geometry and Topology of Caustics, Banach Center Publications, vol.82 (2008), p.71-84.

\bibitem [15] {GW1} Giblin, P.J., Warder, J.P.: \emph{Reconstruction from medial representations}, American Mathematical Monthly 118 (2011), 712--725

\bibitem [16] {GWZ1} Giblin, P.J., Warder, J.P., Zakalyukin, V.M.: \emph{ Bifurcations of affine equidistants}, Proceedings of the Steklov Institute of Mathematics 267 (2009), 57--75.

\bibitem [17] {GZ1} Giblin, P. J., Zakalyukin, V. M.: \emph{Singularities of Centre Symmetry Sets}. Proc. London Math. Soc. (3) 90 (2005), 132--166.

\bibitem [18] {G1} G\'o\'zd\'z, S.: \emph{An antipodal set of a periodic function}, J. of Math. Anal. and App. Vol. 148, pp. 11--21, May 1990.

\bibitem [19] {GO1} Green, M., Osher, S.: \emph{Steiner polynomials, Wulff flows, and some new isoperimetric inequalities for
convex plane curves}, Asia J. Math., 3(1999), 659--676.

\bibitem [20] {G4} Groemer, H.: \emph{Geometric applications of Fourier series and spherical harmonics}, Encyclopedia of Mathematics and its Applications, vol. 61. Cambridge University
Press, Cambridge (1996)

\bibitem [21] {HL1} Hoschek, J., Lasser D.: \emph{Fundamentals of Computer Aided Geometric Design}. A. K. Peters Wellesley MA., Ltd.

\bibitem [22] {H3} Hurwitz, A.: \emph{Sur quelque applications geometriques des series Fourier}, Ann. Sci. Ecole Normal Sup. (3) 19 (1902). 357--408.

\bibitem [23] {H2} Hsiung, C. C.: \emph{A First Course in Differential Geometry. Pure Appl. Math.}, Wiley, New York 1981. 

\bibitem [24] {JJR1} Janeczko, S., Jelonek, Z., Ruas, M. A. S.: \emph{Symmetry defect of algebraic varieties}, Asian J. Math. Vol. 18, No. 3, pp. 525-544, July 2014.

\bibitem [25] {KGP1} Kimia, B., Giblin, P.J., Pollitt, A.: \emph{Transitions of the 3D medial axis under a one-parameter family of deformations}. IEEE Transactions in Pattern Analysis and Machine Intelligence 31 (2009), 900--918.

\bibitem [26] {L1} Lawlor, G.: \emph{A new area-maximization proof for the circle}. Math. Intell. 20 (1998), 29--31

\bibitem [27] {L2} Laugwitz, D.: \emph{Differential and Riemannian Geometry}, Academic Press, New York/ London, 1965.

\bibitem [28] {MM1} Miernowski, A., Mozgawa, W.: \emph{Isoptics of rosettes and rosettes of constant width}, Note di Matematica Vol. 15 - n. 2, 203--213 (1995)

\bibitem [29] {MM2} Miernowski, A., Mozgawa, W.: \emph{On the curves of constant relative width}. Rend. Semin. Mat. Univ. Padova 107, 57–65 (2002)

\bibitem [30] {Moz1} Mozgawa, W.: \emph{Mellish theorem for generalized constant width curves}, Aequat. Math. 89 (2015), 1095–1105, DOI 10.1007/s00010-014-0321-3

\bibitem [31] {PP1} Peternell M., Pottmann H.: \emph{A Laguerre Geometric Approach to Rational Offsets}. Computer Aided Geometric Design vol. 15, pp. 223--249.

\bibitem [32] {RZ1} Reeve, G. M., Zakalyukin, V. M.: \emph{Singularities of the Minkowski set and affine equidistants for a curve and a surface}. Topology Appl. 159 (2012), no. 2, 555--561.

\bibitem [33] {San} Santalo, L.: \emph{Integral geometry and geometric probability}, Encyclopedia of Mathematics and its Applications, Reading, Mass., 1976.

\bibitem [34] {S1} Steiner, J.: \emph{Sur le maximum et le minimum des figures dans le plan, sur la sph\'ere, et dans l'espace en g\'en\'eral}, I and II. J. Reine Angew. Math. (Crelle) 24 (1842), 93-152 and 189-250.

\bibitem [35] {T1} Thom, R.: \emph{Structural stability and morphogenesis}. Reading, Mass: Benjamin (1975).

\bibitem [36] {Z1} Zakalyukin, V.M.: \emph{Envelopes of families of wave fronts and control theory}, Proc. Steklov
Math. Inst. 209 (1995), 133-142.

\bibitem [37] {Zw1} Zwierzy\'nski, M.: \emph{The improved isoperimetric inequality and the Wigner caustic of planar ovals}, J. Math. Anal. Appl. (2016), http://dx.doi.org/j.jmaa.2016.05.016

\bibitem [38] {Zw2} Zwierzy\'nski, M.: \emph{The Constant Width Measure Set, the Spherical Measure Set and isoperimetric equalities for planar ovals}, arXiv:1605.02930

\end{thebibliography}

\end{document}